\documentclass[11pt,reqno]{amsart}

\usepackage[hmargin=4cm,vmargin=4cm]{geometry}

\usepackage[leqno]{amsmath}

\usepackage{tikz}
\usepackage{tikz-cd}
\usetikzlibrary{arrows}

\usepackage{wasysym}

\usepackage{comment}

\theoremstyle{plain}
\newtheorem{proposition}{Proposition}
\newtheorem{theorem}{Theorem}
\newtheorem{corollary}{Corollary}
\newtheorem{lemma}{Lemma}

\theoremstyle{definition}
\newtheorem{definition}{Definition}

\newtheorem{remark}{Remark}

\newtheorem{ex}{Example}

\newcommand{\abs}[1]{|{#1}|}

\newcommand\xqed[1]{%
  \leavevmode\unskip\penalty9999 \hbox{}\nobreak\hfill
  \quad\hbox{#1}}
\newcommand\qer{\xqed{$\fullmoon$}}

\begin{document}

\title[Topological entropy of Turing complete dynamics]{Topological entropy of Turing complete dynamics}

\author[R. Bruera]{Renzo Bruera}
\address{Renzo Bruera, Universitat Politècnica de Catalunya, Departament de Matem\`{a}tiques, Avinguda Diagonal, 647, 08028 Barcelona. \it{e-mail: renzo.bruera@upc.edu}}
\thanks{Renzo Bruera is supported by the Spanish grants FPU20/07006 and PID2021-123903NB-I00, funded by MCIN/AEI/10.13039/501100011033, and by the Catalan grant 2021-SGR-00087.}

\author[R. Cardona]{Robert Cardona}\address{Robert Cardona, Departament de Matem\`atiques i Inform\`atica, Universitat de Barcelona, Gran Via de Les Corts Catalanes 585, 08007 Barcelona, Spain. CRM Centre de Recerca Matem\`{a}tica, Campus de Bellaterra
Edifici C, 08193 Bellaterra, Barcelona, Spain. \it{e-mail: robert.cardona@ub.edu}}

\thanks{Robert Cardona acknowledges partial support from the AEI grant PID2023-147585NA-I00, the Departament de Recerca i Universitats de la Generalitat de Catalunya (2021 SGR 00697). Robert Cardona and Eva Miranda are partially supported by the Spanish State Agency through the Severo Ochoa and Mar\'{\i}a de Maeztu Program for Centers and Units
of Excellence in R\&D (project CEX2020-001084-M)}

\author[E. Miranda]{Eva Miranda}\address{Eva Miranda,
Laboratory of Geometry and Dynamical Systems $\&$ Institut de Matem\`atiques de la UPC-BarcelonaTech (IMTech),  Universitat Polit\`{e}cnica de Catalunya,  Avinguda del Doctor Mara\~{n}on 44-50, 08028, Barcelona.  CRM Centre de Recerca Matem\`{a}tica, Campus de Bellaterra
Edifici C, 08193 Bellaterra, Barcelona
 \it{e-mail: eva.miranda@upc.edu }
 }
\thanks{Eva Miranda is supported by the Catalan Institution for Research and Advanced Studies via an ICREA Academia Prize 2021 and  partially supported by the Spanish State Research agency grant PID2023-146936NB-I00 funded by MICIU/AEI/
10.13039/501100011033, by ERDF/EU. Eva Miranda and Daniel Peralta-Salas are financed under the AEI-DFG project AQUACELL: Celestial Mechanics, Hydrodynamics, and Turing Machines AEI-DFG: Himmelsmechanik, Hydrodynamik und Turing-Maschinen with codes PCI2024-155042-2 and PCI2024-155062-2}
\author[D. Peralta-Salas]{Daniel Peralta-Salas} \address{Daniel Peralta-Salas, Instituto de Ciencias Matem\'aticas, Consejo Superior de Investigaciones Cient\'ificas,
28049 Madrid, Spain. \it{e-mail: dperalta@icmat.es} }
\thanks{Corresponding author. Daniel Peralta-Salas is supported by the grants CEX2023-001347-S, RED2022-134301-T and PID2022-136795NB-I00 funded by MCIN/AEI/10.13039/501100011033. R.C., E.M. and D.P.-S. acknowledge partial support from the grant “Computational, dynamical and geometrical
complexity in fluid dynamics”, Ayudas Fundaci\'on BBVA a Proyectos de Investigaci\'on Cient\'ifica~2021.}

\author[V. Salo]{Ville Salo}\address{Ville Salo, Department of Mathematics and Statistics, University of Turku, 20014 Turun yliopisto, Finland}

\begin{abstract}
We explore the relationship between Turing completeness and topological entropy of dynamical systems. We first prove that a natural class of Turing machines that we call ``branching Turing machines'' (which includes most of the known examples of universal Turing machines) has positive topological entropy. {Motivated by the recent construction of Turing complete Euler flows,} we deduce that any Turing complete dynamics with a continuous encoding that simulates a universal branching machine is chaotic. On the other hand, we show that, unexpectedly, universal Turing machines with zero topological entropy (and even zero speed) can be constructed, {unveiling the independence of chaos and universality at the symbolic level.}
\end{abstract}

\maketitle
\section{Introduction}

Turing machines are one of the most popular models of computation and can be understood as dynamical systems on the space of bi-infinite sequences. In two breakthrough works~\cite{Mo1,Mo2}, Cris Moore introduced the idea of simulating a Turing machine using continuous dynamical systems in the context of classical mechanics. His rough idea was to embed both the space of sequences and the discrete dynamics of the Turing machine into the phase space of a vector field and its continuous dynamics, respectively, using suitable encoding functions. The well-known existence of Turing machines that are universal, i.e., that can simulate any other Turing machine, led him to introduce the definition of a Turing complete dynamical system.

A striking corollary of his ideas was the discovery of a new type of complexity in dynamics, stemming from the fact that certain decision problems, such as the halting problem, cannot be decided by a universal Turing machine.
This yields the conclusion that any Turing complete dynamical system {(see Definition \ref{def:TCDS})} exhibits undecidable behavior for the problem of determining whether the trajectories whose initial data belong to a certain computable set will enter a certain computable open set or not after some time. A priori this computational complexity is very different from the standard way one measures complexity in dynamics, specifically the positivity of the topological entropy, which leads to the usual chaotic behavior.

Since Moore's foundational work, many authors have revisited the problem of simulating Turing machines with dynamical systems in different contexts. This includes low dimensional dynamics~\cite{KCG}, polynomial vector fields~\cite{GCB}, closed-form analytic flows~\cite{KM,GZ}, potential well dynamics~\cite{T1} and fluid flows~\cite{CMPP2,CMP2,CMP3, CMPP3}. As explained in the previous paragraph, these constructions yield complex dynamical systems with undecidable orbits, which are able to simulate any computer algorithm. It is then natural to ask about the relationship between computational complexity and topological entropy, or more precisely: does every universal Turing machine have positive topological entropy? Is every Turing complete dynamical system chaotic? The answer to this question is necessarily delicate, as illustrated by the recent construction of a Turing complete vector field of class $C^\infty$ on $\mathbb S^2$ with zero topological entropy~\cite{CMP3}. 
\medskip

Our goal in this article is {twofold.} {First, we are interested in exploring sufficient computable} conditions under which a universal Turing machine is chaotic in the usual sense of positive topological entropy, {and conditions under which this property holds for a Turing complete dynamical system simulating such a universal Turing machine.} {Secondly, we aim to improve our understanding of the connections between these two different notions of complexity: positive topological entropy and computational universality. As we will see, we establish that these are independent already at the symbolic level: namely, after showing that most examples of universal Turing machines are chaotic, we construct universal Turing machines with zero topological entropy.} The study of Turing machines from a dynamical perspective has been developed by several authors, see e.g.~\cite{Kurka, BCN, KO}. 

\medskip

Under suitable conditions, we will be able to construct dynamical systems that exhibit both undecidable paths and chaotic invariant sets of horseshoe type. Our construction is based on the notion of branching Turing machine, which is presented in Section~\ref{SS.crit}. Our {first contribution is a computable criterion (the ``branching'' property) for a Turing machine to have positive topological entropy.} Positivity of the entropy of a Turing machine was characterized by Jeandel \cite{J}. However, unlike our criterion, this characterization is not computable because determining whether a one-tape Turing machine has positive topological entropy or not is undecidable in general~\cite{GOTA}.

\begin{theorem}\label{T.main}
Any branching Turing machine has positive topological entropy.
\end{theorem}

{Since the criterion is computable, it can be used to identify particular examples of universal Turing machines with positive entropy. Moreover, we are not aware of any examples in the literature of non-branching universal Turing machines}. We give in Section~\ref{appendix}  a tour-de-force construction of such a machine (indeed, one with zero entropy). 
In Section~\ref{S.app}, we relate the topological entropy of a Turing machine with the topological entropy of a dynamical system that simulates it (in some precise sense, in particular with a continuous encoding). Other works that have analyzed the relationship between the topological entropy (and its computability) of symbolic systems and more general dynamical systems are~\cite{R1, R2}. 

Combining Theorem~\ref{T.main} with the construction of Turing complete diffeomorphisms of the disk presented in~\cite{CMPP2} and the continuity of the encodings used there, we easily get the following corollary:

\begin{corollary}
{Any Turing complete smooth area-preserving diffeomorphism of the disk $\varphi: D\rightarrow D$ as constructed in~\cite{CMPP2} necessarily has positive topological entropy whenever the simulated universal Turing machine is branching.}
\end{corollary}


The technique introduced in~\cite{CMPP2}, which is based on the suspension of Turing complete area-preserving diffeomorphisms of the disk, immediately yields the construction of Reeb flows on $\mathbb S^3$ {whose chaotic behavior is already inherited from its Turing universality}.  As argued in~\cite{CMPP2} there are many compatible Riemannian metrics $g$ that make these Reeb flows stationary solutions of the Euler equations on $(\mathbb S^3,g)$. Incidentally, these metrics cannot be optimal (or critical) because the Reeb flows have positive topological entropy~\cite{MPS}.

{As we anticipated, our last result is that, surprisingly, there exist universal Turing machines with zero topological entropy (in fact, with zero speed, which is a stronger condition). It can be required to be reversible too.
\begin{theorem}\label{T.main2}
    There exists a universal Turing machine (which can be required to be reversible) that has zero speed. In particular, it has zero topological entropy.
\end{theorem}
}

This article is organized as follows. In Section~\ref{S.TM} we recall the usual interpretation of Turing machines as dynamical systems on compact metric spaces and prove some auxiliary results. In Section~\ref{S.entro} we define the topological entropy of a Turing machine and show how it is related to the usual definition of topological entropy in dynamics. Theorem~\ref{T.main} is proved in Section~\ref{SS.crit}, its corollary in Section~\ref{S.app} and the example of a non-branching, zero-entropy Turing machine (Theorem~\ref{T.main2}) is presented in Section~\ref{appendix}. 

\section{Turing machines as dynamical systems}\label{S.TM}

In this section, we explain how to define a continuous dynamical system on a compact metric space using a Turing machine, and other notions related to Turing machines.

\subsection{The global transition function}\label{SS.gtf}

A Turing machine $T=(Q,q_0,q_{halt},\Sigma,\delta)$ is defined by:
\begin{itemize}
\item A finite set $Q$ of ``states'' including an initial state $q_0$ and a halting state $q_{halt}$.
\item A finite set $\Sigma$ which is the ``alphabet'' with cardinality at least two. It has a special symbol, denoted by $0$, that is called the blank symbol.
\item A transition function $\delta:Q\setminus \{q_{halt}\}\times \Sigma \longrightarrow Q\times \Sigma \times \{-1,0,1\}$.
\end{itemize}
{For technical reasons, for each Turing machine, we assume that its state set is disjoint from its alphabet set, and in fact from the alphabet sets of any other Turing machine}

The evolution of a Turing machine is described by an algorithm. At any given step, the configuration of the machine is determined by the current state $q\in Q$ and the current tape $t=(t_n)_{n\in \mathbb{Z}}\in \Sigma^\mathbb{Z}$. The pair $(q,t)$ is called a configuration of the machine. Any real computational process occurs throughout configurations such that every symbol in the tape $t$ is $0$ except for finitely many symbols. A configuration of this type will be called compactly supported.

The algorithm is initialized by setting the current configuration to be $(q_0,s)$, where $s=(s_n)_{n\in \mathbb{Z}}\in \Sigma^{\mathbb{Z}}$ is the input tape. Then the algorithm runs as follows:

\begin{enumerate}
\item Set the current state $q$ as the initial state and the current tape $t$ as the input tape.
\item If the current state is $q_{halt}$, then halt the algorithm and return $t$ as output. Otherwise, compute $\delta(q,t_0)=(q',t_0',\varepsilon)$, with $\varepsilon \in \{-1,0,1\}$.
\item Replace $q$ with $q'$, and change the symbol $t_0$ by $t_0'$, obtaining the tape $\tilde t=...t_{-1}.t_0't_1...$ (as usual, we write a point to denote that the symbol at the right of that point is the symbol at position zero).
\item Shift $\tilde t$ by $\varepsilon$ obtaining a new tape $t'$, then return to step $(2)$ with the current configuration $(q',t')$. Our convention is that $\varepsilon=1$ (resp. $\varepsilon=-1$) corresponds to the left shift (resp. the right shift).
\end{enumerate}

Given a Turing machine $T$, its transition function can be decomposed as
$$\delta=(\delta_Q,\delta_\Sigma,\delta_\varepsilon) : Q\times \Sigma \rightarrow Q\times \Sigma\times \{-1,0,1\}\,.$$
Here the maps $\delta_Q,\delta_\Sigma$ and $\delta_\varepsilon$ denote the composition of $\delta$ with the natural projections of $Q\times \Sigma\times \{-1,0,1\}$ onto the corresponding factors. The Turing machine can be understood as a dynamical system $(R_T,X_T)$, where the phase space is
$$X_T:=Q\times \Sigma^{\mathbb{Z}}$$
and the action
$$R_T:X_T\rightarrow X_T$$
is the \emph{global transition function}, which is given by
\begin{equation}\label{turing_Def}
	R_T(q,(\dots t_{-1}.t_0t_1\dots)):=
	\begin{cases}
		(q',(\dots t_{-1}t_0'.t_1\dots))\,, &\text{ if }\delta_{\varepsilon}(q,t_0)=1\,,\\
		(q',(\dots t_{-2}.t_{-1}t_0'\dots))\,, &\text{ if }\delta_{\varepsilon}(q,t_0)=-1\,,\\
		(q',(\dots,t_{-1}.t_0't_1\dots))\,, &\text{ if }\delta_{\varepsilon}(q,t_0)=0\,,\\
	\end{cases}
\end{equation}
for $q\neq q_{halt}$, with $t_0'=\delta_\Sigma(q,t_0)$ and $q'=\delta_Q(q,t_0)$. For $q=q_{halt}$, several extensions of the global transition function exist, the simplest and most natural being
\begin{equation}\label{turing_Def2}
R_T(q_{halt},(\dots t_{-1}.t_0t_1\dots)):=(q_{halt},(\dots t_{-1}.t_0t_1\dots))\,.
\end{equation}
This is equivalent to extending the transition function on halting configurations as $\delta(q_{halt},t_0):=(q_{halt},t_0,0)$; all along this article we shall assume that the transition function is extended this way. We will use the notation 
$$X_T^c =\{\text{compactly supported configurations}\}\,,$$ 
that is, the set of configurations of $T$ with a tape that has only finitely many non-blank symbols.

\subsection{Properties of $X_T$ and $R_T$}
Given a Turing machine $T$, for each $x=(q,t)\in X_T$ let us set $x_q:= q$ and $x_i:= t_i$. If we endow the finite sets $Q$ and $\Sigma$ with the discrete topology, $X_T$ becomes a compact metric space endowed with the complete metric
\begin{equation}\label{eq.dist}
	d\left(x,x'\right) := \begin{cases}
		1, &\text{ if }x_q\neq x'_q\,,\\
		2^{-n}, &\text{ if }x_q = x'_q \text{ and } n=\operatorname{sup} \{ k : x_i=x'_i \ \forall  |i|<k\}\,.
	\end{cases}
\end{equation}
It is elementary to check the continuity of the global transition function for this metric:
\begin{lemma}\label{L.contR}
The global transition function $R_T:X_T\to X_T$ is continuous for the metric $d$.
\end{lemma}

The metric $d$ defines a topology in the compact space $X_T$. Then, several natural sets and functions are open and continuous for this topology. Particularly important are the halting domain and the halting time:

\begin{definition}[Halting domain and halting time]
Let $T$ be a Turing machine with corresponding global transition function $R_T$ acting on the phase space $X_T$. As usual, the \emph{halting domain} of $T$ is defined as the set
\begin{equation}
	X^{H}_T:=\{x\in \{q_0\}\times \Sigma^{\mathbb{Z}} : \exists N\in \mathbb{N} \text{ such that } {R_T^N}\left(x\right)_q = q_{halt}\}\,.
\end{equation}
The first number $N\in \mathbb{N}$ such that ${R_T^N}(x)_q=q_{halt}$ is called the \emph{halting time} of $x$.
\end{definition}

The following lemma is standard and shows that the halting domain $X^H_T$ is open in $X_T$ and the halting time $N$ is continuous (both properties with respect to the natural topology on $X_T$ introduced before). For the benefit of a non-expert reader we provide a proof.

\begin{lemma}\label{L:contN}
The halting domain of a Turing machine $T$ is open in $X_T$, and the halting time function $N:X^H_T\to \mathbb{N}$ is continuous.
\end{lemma}
\begin{proof}
To see that $X^H_T$ is open we simply notice that
\begin{equation*}
X^H_T = \left(\bigcup_{n\geq 1} \{x\in X_T: R_T^n\left(x\right)_q = q_{halt}\}\right)\cap \left(\{q_0\}\times \Sigma^{\mathbb{Z}}\right)\,.
\end{equation*}
Since the first term is the union of open sets, and the second term is also open, the claim follows.

To see the continuity of $N$, we observe that for a fixed $x\in X_T$, since $n:=N(x)$ is finite, the machine can read (at most) the cells of the sequence $x$ in positions $[-n,n]$. Therefore, any other input coinciding with $x$ in the range $[-n,n]$ will halt after the same number of steps. We then conclude that for each natural number $n$, the set $N^{-1}(n)\subseteq X_T$ contains an open ball of radius $2^{-n}$ and center $x$, which implies the continuity of $N$.
\end{proof}

It is also convenient to introduce the notion of output function of a Turing machine, which assigns to the halting domain $X^H_T$ the set of elements in $X_T$ that are reached when halting. Next, we introduce the precise definition and the key property that the output function is continuous.

\begin{definition}[Output function]
For each Turing machine $T$, we define the output function $\Psi_T:X^{H}_T\to X_T$ as
\begin{equation}
	\Psi_T(x):=R_T^{N(x)}(x) \quad \forall x\in X^H_T\,,
\end{equation}
where $N(x)$ is the halting time function.
\end{definition}

\begin{lemma}\label{L:output}
The output function is continuous.
\end{lemma}
\begin{proof}
Fix $z\in X^H_T$. Since the halting time $N$ is continuous, then it is locally constant, so there is a ball $B(z)\subset X^H_T$ around $z$ where $N(x)=N(z)$ for all $x\in B(z)$. Therefore, locally we can write $\Psi_T(x)=R_T^{n}(x)$ for all $x \in B(z)$, with $n:=N(z)$. Since $R_T$ is continuous (cf. Lemma~\ref{L.contR}), it follows that $\Psi_T$ is continuous in $B(z)$ for all $z$, so $\Psi_T$ is a continuous function.
\end{proof}

\subsection{Universal Turing machines}
Finally, we introduce a definition of \emph{universal Turing machine (UTM)} following Morita in~\cite{Mor} (which is based on~\cite{Rog}). We remark that this definition is more general than the classical ones used by Shannon~\cite{Sha} and Minsky~\cite{Min} in their foundational works.

Let $\{T_{n}\}_{n=1}^\infty$ be an enumeration of all Turing machines and define the union
$$\mathcal{X}:=\bigcup_{n\geq 1}X_{T_n}^c$$
of all compactly supported configurations of $T_n$. {This space is countable, and elements of $\mathcal{X}$ are naturally in bijection with finite words of the form $\Sigma^* Q \Sigma^*$ (here, $\Sigma^*$ denotes the finite words over $\Sigma$), where neither the first nor last symbol is blank. Specifically, the word $uqv$ where $u,v \in \Sigma^*$ and $q \in Q$ denotes the configuration $(q, \ldots 000 u . v 000 \ldots)$. For all computational purposes, we think of $x \in X_{T_n}^c$ as being represented by such a word.}

\begin{definition}[Universal Turing machine]\label{def:UTM}
     A Turing machine $U$ is universal if there exist total computable functions $c:\mathbb{N}\times \mathcal{X}\to X_U^c$ and $d:X_U^c\to \mathcal X$ such that for each $n\in \mathbb{N}$,
\begin{equation}
	\Psi_{T_n}(x) = (d\circ \Psi_U \circ c(n,\cdot ))(x) \quad \forall x \in X_{T_n}^c\,.
\end{equation}
\end{definition}

{Recall that a total computable function is simply a function that is itself computed by a Turing machine.}

 
Most examples of universal Turing machines satisfy Definition~\ref{def:UTM}; examples can be found in \cite{Sha, Min, Mor}. There are also some examples of Turing machines in the literature, which have been shown universal in some weaker sense involving non-compact support configurations, see for example \cite{Sm}. A particularly well-known property of universal Turing machines is that the halting problem for compactly supported inputs is undecidable for these machines.

\section{Topological entropy of Turing machines}\label{S.entro}

The topological entropy of Turing machines has been studied in~\cite{Op} from a dynamical systems viewpoint, and its computability was analyzed by Jeandel in~\cite{J}, by working with the entropy as is done with the speed of the machine. In this section, we recall Oprocha's formula to compute the topological entropy of a Turing machine and introduce Moore's generalized shifts as a model to describe the dynamics of any Turing machine. As we will see, the topological entropy of the generalized shift coincides with that of the Turing machine it simulates.

Let $T$ be a Turing machine. As argued in Section~\ref{S.TM}, it can be described using the global transition function $R_T$, which is a continuous dynamical system on the compact metric space $(X_T,d)$. In this setting, one can use the definition of topological entropy given by Bowen and Dinaburg~\cite{Bo,Din}, which is equivalent to the original one by Adler, Konheim, and McAndrew~\cite{Ad}. 

In~\cite[Theorem~3.1]{Op} Oprocha obtained a remarkable formula showing that the topological entropy of $T$ can be computed as the following limit:
\begin{equation}\label{eq.opr}
h(T) := \lim_{n\to \infty} \frac{1}{n}\log |S(n,R_T)|\,,
\end{equation}
where $|\cdot|$ denotes the cardinality of the finite set $S(n,R_T)$, which is defined as
\begin{equation}
S(n,R_T):= \{u\in (Q\times \Sigma)^n : \exists x\in X_T \text{ s.t. } u_i = (R_T^{i-1}(x)_q,R_T^{i-1}(x)_0)\,\}\,,
\end{equation}
where $u_i$ is defined for $i=1,...,n$ and denotes the $i^{\text{th}}$ component of $u$.
Here, $R_T^j$ denotes the $j$-th iterate of the map $R_T$. Usually, the set $S(n,R_T)$ is called the set of $n$-words allowed for the Turing machine $T$.

In~\cite{Mo2} Moore introduced a generalization of the shift map that he called a \emph{generalized shift}, which is a class of dynamical systems that allows one to describe any Turing machine, and is different from the global transition function $R_T$. Let us introduce Moore's idea and how it connects with the dynamics and topological entropy of $(R_T,X_T)$. These results will be key to prove the existence of Turing complete diffeomorphisms of the disk with positive topological entropy, cf. Section~\ref{S.app}.

\begin{definition}[Generalized shift]
Let $A$ be a finite set. For each $s\in A^\mathbb{Z}$ and $J,J'\in\mathbb Z$, we denote by $s_{[J,J']}$ the finite string containing the elements of $s$ in positions $J$ to $J'$, and we denote by $\oplus$ the operation of string concatenation. A generalized shift is a map $\Delta : A^\mathbb{Z}\to A^\mathbb{Z}$ that is given by
\begin{equation*}
\Delta(s) = \sigma^{F\left(s_{[-r,r]}\right)}\left(\dots s_{(-\infty,-r-1]}\oplus G(s_{[-r,r]})\oplus s_{[r+1,\infty)}\right)\,.
\end{equation*}
Here, $r$ is a natural number, $\sigma$ is the Bernoulli shift and $F$ and $G$ are maps
\begin{align*}
	F &: A^{2r+1}\to \mathbb{Z}\,,\\
	G &: A^{2r+1}\to A^{2r+1}\,.
\end{align*}
\end{definition}

As in Section~\ref{SS.gtf}, if we endow $A$ with the discrete topology then $A^{\mathbb{Z}}$ is a compact metric space and the generalized shift $\Delta$ is always continuous (independently of the choice of $r$, $F$, and $G$). We observe that the complete metric is defined as in the second formula of Equation~\eqref{eq.dist}.

\begin{remark}\label{R.Cantor}
Without any loss of generality one may assume that $A = \{0,1\}$, and thus $A^{\mathbb{Z}}$ is homeomorphic to the square ternary Cantor set $C^2\subset \mathbb{R}^2$ via the homeomorphism given by
\begin{align}\label{eq:coding}
e(s) = \left(\sum_{k=1}^{\infty} s_{-k}\frac{2}{3^k}, \ \sum_{k=1}^{\infty} s_{k-1}\frac{2}{3^k}\right)\,.
\end{align}
Accordingly, any generalized shift can be viewed as a dynamical system on the square Cantor set. \qer
\end{remark}

The connection between Turing machines and generalized shifts was established by Moore~\cite{Mo2}. Let $T$ be a Turing machine with set of states $Q$ and set of symbols $\Sigma$, and define $A=Q\cup \Sigma$. With $X_T=Q\times \Sigma^{\mathbb{Z}}$, we also define the injective map
\begin{align}
\varphi : X_T & \to A^{\mathbb{Z}}\\
(q,t)&\mapsto (\dots t_{-1}.qt_0\dots)\,.
\end{align}
Then:
\begin{theorem}[Moore]\label{theorem_moore}
Given a Turing machine $T$, there exists a generalized shift $\Delta$ on $A^{\mathbb Z}$ such that its restriction to $\varphi(X_T)$ satisfies
\begin{equation}
	\Delta\vert_{\varphi(X_T)} = \varphi \circ R_T \circ \varphi^{-1}\vert_{\varphi(X_T)}.
\end{equation}
\end{theorem}

In fact, the map $\varphi : X_T \to A^{\mathbb{Z}}$ is more than injective, it is a topological conjugation between the dynamical systems $\Delta$ and $R_T$. This will allow us to relate both dynamics. The following lemma is standard, but we provide a proof for the benefit of a non-expert reader.

\begin{lemma}\label{L:homeo}
The map $\varphi$ is a homeomorphism onto its image.
\end{lemma}
\begin{proof}
Let us see that $\varphi$ is continuous. Indeed, for any $\varepsilon>0$ choose $k$ such that $\varepsilon>2^{-k}$. Let $x,x'\in X_T$ be such that $d(x,x')<2^{-k}$. This means that
\begin{equation*}
x_q = x'_q \text{ and } x_i=x'_i \ \text{ for all } \abs{i}<k\,.
\end{equation*}
Then, clearly
\begin{equation*}
d\left(\varphi(x),\varphi(x')\right) < 2^{-k} < \varepsilon\,,
\end{equation*}
thus implying continuity.

Similarly we may show that $\varphi^{-1}\vert_{\varphi(X_T)}$ is continuous. Fix $\varepsilon>0$, let $k$ be such that $\varepsilon>2^{-k}$ and let $y,y'\in \varphi(X_T)\subset A^{\mathbb{Z}}$ such that $d(y,y')<2^{-(k+1)}$. That is, $y_i=y'_i$ for all $\abs{i}<k+1$. Note that since $y$ and $y'$ are in $\varphi(X_T)$, they are of the form
\begin{equation*}
y = y_{(-\infty,-1]}\oplus y_0\oplus y_{[1,\infty)}
\end{equation*}
with $y_0\in Q$ and $y_i\in \Sigma$ for all $i\neq 0$. Then, clearly
\begin{equation*}
d\left(\varphi^{-1}(y),\varphi^{-1}(y')\right) < 2^{-k}\,,
\end{equation*}
and the lemma follows.
\end{proof}

For the purposes of this article, the main consequence of the aforementioned conjugation between Turing machines and generalized shifts is the following result, which shows the connection between the topological entropy of both systems. Since a generalized shift $\Delta$ is a map on a compact metric space, its topological entropy $h(\Delta)$ can be computed using Bowen-Dinaburg's definition, as before.

\begin{proposition}\label{prop:entropyGS}
Let $T$ be a Turing machine and $\Delta$ its associated generalized shift. Then $h(\Delta)\geq h(T)$. In particular, if $T$ has positive topological entropy, then so does its associated $\Delta$.
\end{proposition}
\begin{proof}
Since $\Delta\vert_{\varphi(X_T)} = \varphi \circ R_T \circ \varphi^{-1}\vert_{\varphi(X_T)}$, cf. Theorem~\ref{theorem_moore}, it is clear that the subset $\varphi(X_T)\subset A^{\mathbb{Z}}$ is forward invariant under the iterations of the generalized shift $\Delta$. Moreover, this property and Lemma~\ref{L:homeo} also show that the maps $R_T$ and $\Delta$ are topologically conjugate via the homeomorphism $\varphi:X_T\to \varphi(X_T)$. The invariance of the topological entropy under homeomorphisms and the fact that
\[
h(\Delta)\geq h(\Delta|_{\varphi(X_T)})\,,
\]
see e.g.~\cite[Section 3.1.b]{KH}, complete the proof of the proposition.
\end{proof}

%

\section{A criterion for positive topological entropy}\label{SS.crit}

In this section, we prove our first result, which shows that any Turing machine that belongs to a special class (that we call branching) has positive topological entropy. To this end, we exploit the fact that a dynamical system has positive topological entropy if it exhibits an invariant subset where the entropy is positive.

\subsection{Positive entropy for strongly branching Turing machines}
To illustrate the method of proof, we first consider the class of strongly branching Turing machines:
\begin{definition}[Strongly branching Turing machine]
A Turing machine $T$ is \emph{strongly branching} if for some $\varepsilon\in\{-1,1\}$ there exists a subset $Q'\times \Sigma' \subset (Q\setminus \{q_{halt}\})\times \Sigma$, with $|\Sigma'|\geq 2$, such that $\delta_Q(Q'\times \Sigma')\subseteq Q'$ and {$\delta_{\varepsilon}\vert_{Q'\times \Sigma'} \equiv \varepsilon$}. 
\end{definition}
Given a strongly branching Turing machine $T$ with $\varepsilon=1$ (the case $\varepsilon=-1$ is analogous), we claim that the subset
\begin{equation*}
	Y_T := \{x\in X_T : x_q\in Q', \ x_i \in \Sigma' \ \text{ for all } i\geq 0\} = Q'\times \Sigma^{\mathbb{N}_0}\times \Sigma'^{\mathbb{N}}\subset X_T
\end{equation*}
is forward invariant under the global transition function $R_T$. Here $\mathbb N_0$ is the set of natural numbers without $0$.
\begin{lemma}
$Y_T$ is forward invariant under $R_T$.
\end{lemma}
\begin{proof}
Indeed, let $x=(q,(t_i))\in Y_T$. We have
\begin{equation*}
	R_T(x) = (q',(\dots t_{-1}t_0'.t_1\dots)) = (q',(\dots s_{-1}.s_0s_1\dots))
\end{equation*}
with $s_i:=t_{i+1}\in \Sigma'$ for all $i\geq 0$ and $q'=\delta_Q(q,t_0)\in Q'$ by hypothesis. Therefore, $R_T(x)\in Y_T$ as claimed.
\end{proof}

This lemma allows us to prove the following sufficient condition for positive topological entropy.
\begin{theorem}\label{pos_entropy_turing}
Let $T$ be a strongly branching Turing machine. Then
$$h(T)\geq \log |\Sigma'|>0\,.$$
\end{theorem}
\begin{proof}
As before, let us consider that $T$ is strongly branching with $\varepsilon=1$, the other case being completely analogous. To estimate the topological entropy of $T$ we use Oprocha's formula in Equation~\eqref{eq.opr}. We claim that for each $n\geq 1$, $|S(n,R_T)|\geq |\Sigma'|^n$. To see this, fix $n$, let $q^0\in Q'$ and consider any finite sequence $\{a_0',a_1',\dots,a_{n-1}'\}\subset \Sigma'$. Choose any $x\in X_T$ such that $x_q = q^0$ and $x_i = a_i'$ for $i=0,\dots,n-1$. We define $q^i = R_T^i(x)_q$ for $i=1,\dots,n-1$ and finally we set $u=\left( (q^0,a'_0),\dots,(q^{n-1},a'_{n-1})\right)\in (Q'\times \Sigma')^n$. Since $\delta_\varepsilon\vert_{Q'\times \Sigma'}=1$, from~\eqref{turing_Def} we infer that
\begin{align*}
	(R_T^0(x)_q,R_T^0(x)_0)&=(q^0,t_0)=(q^0,a'_0)=u_0\,,\\
	(R_T^1(x)_q,R_T^1(x)_0)&=(q^1,t_1)=(q^1,a'_1)=u_1\,,\\
	&\vdots \\
	(R_T^{n-1}(x)_q,R_T^{n-1}(x)_0)&=(q^{n-1},t_{n-1})=(q^{n-1},a'_{n-1})=u_{n-1}\,,
\end{align*}
so $u\in S(n,R_T)$. Since this holds for all finite sequences of length $n$ in $\Sigma'$, we conclude that $\abs{S(n,R_T)}\geq \abs{\Sigma'}^n$. Hence
\begin{equation*}
	h(T) \geq \lim_{n\to\infty}\frac{1}{n}\log \abs{\Sigma'}^n = \log \abs{\Sigma'}\,,
\end{equation*}
as we wanted to prove.
\end{proof}

This theorem can be readily applied to show that some particular examples of universal Turing machines exhibit positive topological entropy. {In the following, we will refer to well-known examples of ``small" universal Turing machines. The notation used, which is standard, is $UTM(m,k)$ for a (concrete) universal Turing machine with $m$ states and $k$ alphabet symbols. Similarly, $URTM(m,k)$ denotes a reversible universal Turing machine, and $WUTM(m,k)$ a weakly universal Turing machine (we refer to \cite{Near} for a definition of weak universality).} For instance, the machine $T$ denoted as $UTM(6,4)$ in~\cite{Near} has a transition function $\delta$ specified by the following table. The horizontal axis contains the states and the vertical axis contains the symbols. Here, $L$ and $R$ stand for $\delta_\varepsilon=-1$ and $\delta_\varepsilon = 1$ in our notation.

\begin{center}
\begin{figure}[!h]
\begin{tabular}{ c|c c c c c c }
 $U_{6,4}$ & $u_1$ & $u_2$ & $u_3$ & $u_4$ & $u_5$ & $u_6$ \\
  \hline
 $g$ & $u_1bL$ & $u_1gR$ & $u_3bL$ & $u_2bR$ & $u_6bL$ & $u_4bL$ \\
 $b$ & $u_1gL$ & $u_2gR$ & $u_5bL$ & $ u_4gR $ & $u_6gR $ & $u_5 gR $  \\
 $\delta$ & $u_2cR$ & $u_2cR$ & $u_5\delta L$ & $u_4cR $ & $u_5\delta R$ & $u_1gR$ \\
 $c$ & $u_1\delta L$ & $u_5gR$ & $u_3\delta L$ & $u_5cR $ & $u_3 b L$ & halt
\end{tabular}
\caption{Transition table of $UTM(6,4)$.}
\end{figure}
\end{center}

It is clear that $Q'\times \Sigma' := \{u_2\}\times \{b,\delta\}$ satisfies the hypotheses in the definition of a strongly branching Turing machine with $\varepsilon=1$, and hence a straightforward application of Theorem~\ref{pos_entropy_turing} yields $h(T)\geq \log 2>0$, that is:
\begin{corollary}
The universal Turing machine $UTM(6,4)$ has positive topological entropy.
\end{corollary}
The same argument works when considering, for example, the reversible universal Turing machine $URTM(10,8)$ as introduced in~\cite[Section 7.3.2]{Mor}:
\begin{corollary}
The universal Turing machine $URTM(10,8)$ has positive topological entropy.
\end{corollary}

Other examples of universal Turing machines that are strongly branching are $UTM(5,5), UTM(4,6)$ and $UTM (10,3)$ in~\cite{Rog}, or $UTM(5,5), UTM(9,3)$ and $UTM(6,4)$ in \cite{Near}. The weakly universal Turing machines $WUTM(3,3)$ and $WUTM(2,4)$ in~\cite{WN}, or the famous Wolfram's weakly universal $(2,3)$ Turing machine \cite{Sm} are also strongly branching. We will later show that the universal Turing machine $UTM(15,2)$ in \cite{Near} or the weakly universal Turing machine $(6,2)$ in~\cite{WN} are not strongly branching but are branching according to the definition in the following subsection (so, in particular, they have positive topological entropy).

\subsection{A generalized criterion for branching Turing machines}

In this subsection we establish a more general version of Theorem~\ref{pos_entropy_turing} that is also computable and implies that the Turing machine has positive topological entropy.

For this criterion we need to introduce some notation. As before, we denote tapes in $\Sigma^{\mathbb{Z}}$ using $t$ and $t_i$ denotes the $i^{\text{th}}$ symbol of $t$. We will construct a computable function
$$\phi: Q\times \Sigma \longrightarrow \{H,P\} \sqcup \{\pm 1\}\times Q,$$
which tells us whether the machine, with current state $q$ and reading the symbol $s$ at the zero position, will eventually (after perhaps some steps without shifting) shift to the right, to the left, or not shift at all before halting (H) or becoming periodic (P). More precisely, given a pair $(q,s)\in Q\times \Sigma$, if $\delta_\varepsilon((q,s))=\pm 1$ and $\delta_Q((q,s))\neq q_{halt}$, then
$$\phi(q,s):=(\delta_{\varepsilon}(q,s), \delta_Q(q,s))\,.$$
If $\delta_Q(q,s)=q_{halt}$ then
$$\phi(q,s):=H\,.$$
Otherwise, setting $\delta_Q(q,s)=q_1$ and $\delta_\Sigma(q,s)=s_1$, if $\delta_Q(q_1,s_1)=q_{halt}$, then we define
$$\phi(q,s):=H\,,$$
and we iterate this process. It is easy to check that for any pair $(q,s)$, only the following possibilities can occur for the aforementioned iteration:

\begin{enumerate}
 \item After $k$ steps of the machine without any shifting, we reach a configuration $(\tilde q, \tilde t)$ such that $\delta_Q(\tilde q, \tilde t_0)=q_{halt}$. In this case $\phi(q,s):=H$.
 \item The iterates of the global transition function applied to $(q,t)$ never shift nor reach a halting configuration. Then for some $k$ we have $R_T^k(q,t)=(q,t)$ and the orbit becomes periodic. We define in this case $\phi(q,s):=P$.
 \item After $k$ steps without any shift, the machine shifts to the left without halting. That is, a configuration of the form $(q_+,\tilde t)$ with $\delta(q_+,\tilde t_0)=(q',s',1)$, $q'\neq q_{halt}$, is reached. In this case, we define $\phi(q,s):=(1, q')$.
 \item After $k$ steps without shifting, the machine shifts to the right without halting. That is, a configuration of the form $(q_-,\tilde t)$ with  $\delta(q_-,\tilde t_0)=(q',s',-1)$ and $q' \neq q_{halt}$ is reached. In this case, we define $\phi(q,s):=(-1, q')$.
\end{enumerate}
Of course, the integer $k$ depends on the pair $(q,s)$, and an upper bound can be obtained in terms of the number of states and alphabet symbols. We can define the function $\tau: Q\times \Sigma \rightarrow \mathbb{Z}$ as the function giving such an integer $k$.

\begin{definition}[Branching Turing machine]\label{def:regular}
A Turing machine $T$ is \emph{branching} if, for some $\varepsilon\in \{-1,1\}$, there exist two different sequences 
\[
(q_1,s_1),...,(q_{m_1},s_{m_1}) \text{ and } (q_1',s_1'),...,(q_{m_2}',s_{m_2}')
\]
of pairs in $Q\backslash\{q_{halt}\}\times \Sigma$, with $m_1\geq 2$, $m_2\geq 2$, such that $q_1=q_1'$, $\phi(q_i,s_i)=(\varepsilon,q_{i+1})$, $\phi(q_j',s_j')=(\varepsilon,q_{j+1}')$ for all $1\leq i\leq m_1-1$, $1\leq j\leq m_2-1$, and $\phi(q_{m_1},s_{m_1})=\phi(q_{m_2}',s_{m_2}')=(\varepsilon,q_1)$. We also require that none of the sequences is a concatenation of copies of the other sequence.
\end{definition}

\begin{remark}
{The definition of branching Turing machine becomes simpler to state if one works with Turing machines such that $\varepsilon\in \{-1,1\}$. We have stuck to the model with $\varepsilon\in \{-1,0,1\}$ to be consistent with previous works, e.g. \cite{CMPP2}.}
\end{remark}
\paragraph{\textbf{Graph interpretation}.} The branching of a Turing machine can be easily understood in terms of two graphs that we can associate to $T$ using the function $\phi$. These graphs are different, although somewhat related, from the classical state diagram of the transition function of a Turing machine, see e.g. \cite[Section 3.1]{Sip}.
For each $\varepsilon\in\{-1,1\}$ we can associate to $T$ a graph as follows: the vertices of the graph are the set of states of the machine. Given two vertices $q$ and $q'$, we define an edge oriented from $q$ to $q'$ for each $s\in\Sigma$ such that $\phi(q,s)=(\varepsilon,q')$. It is then obvious from the definition, that a Turing machine is branching if and only if for some $\varepsilon$ the corresponding graph contains two different oriented cycles with at least one common vertex.

The following examples show that there are universal Turing machines that are branching, but not strongly branching. The first example also illustrates the graph interpretation of a branching Turing machine.

\begin{ex}
An example of a (weakly) universal Turing machine that is branching but not strongly branching is given by the $(6,2)$ machine in \cite{WN}. As before, the notation $WUTM(m,k)$ means that the Turing machine consists of $m$ states and $k$ alphabet symbols.
\begin{center}
\begin{figure}[!h]
\begin{tabular}{ c|c c c c c c }
 $WU_{6,2}$ & $u_1$ & $u_2$ & $u_3$ & $u_4$ & $u_5$ & $u_6$ \\
  \hline
 $g$ & $u_10L$ & $u_60L$ & $u_20R$ & $u_51R$ & $u_41L$ & $u_11L$ \\
 $b$ & $u_21L$ & $u_30L$ & $u_31L$ & $ u_60R $ & $u_41R $ & $u_4 0R $
\end{tabular}
\caption{Transition table of $WUTM(6,2)$.}
\end{figure}
\end{center}

It is easy to check that it is not strongly branching. On the other hand, the sequences $(4,g), (5,b), (4,g)$ and $(4,b), (6,g), (4,b)$ satisfy the required properties for the machine to be branching. Figure \ref{fig:graph} pictures the graph (as defined before) of the machine for $\varepsilon=1$. Notice how $u_4$ belongs to two different cycles.
\begin{figure}[!h]
\begin{tikzpicture}

\tikzset{vertex/.style = {shape=circle,draw,minimum size=1.5em}}
\tikzset{edge/.style = {->,> = latex'}}
\node[vertex] (u1) at  (0,0) {$u_1$};
\node[vertex] (u2) at  (2,2) {$u_2$};
\node[vertex] (u3) at  (4,2) {$u_3$};
\node[vertex] (u4) at  (6,0) {$u_4$};
\node[vertex] (u5) at (4,-2) {$u_5$};
\node[vertex] (u6) at (2,-2) {$u_6$};

\draw[edge] (u3) to (u2);
\draw[edge] (u4) to (u5);
\draw[edge] (u4) to (u6);
\draw[edge] (u5) to[bend right] (u4);
\draw[edge] (u6) to[bend left] (u4);

%
%
%
%
\end{tikzpicture}
\caption{Graph of $WU_{6,2}$ for $\varepsilon=+1$.}
\label{fig:graph}
\end{figure}
\end{ex}

\begin{ex}
The universal Turing machine $UTM(15,2)$ in \cite{Near} is another example of universal Turing machine that is not strongly branching, but it is branching. Looking at the transition table~\cite[Table 16]{Near}, one notices that the sequences $(u_4,c), (u_6,c), (u_4,c)$ and $(u_4,c), (u_6,b), (u_4,c)$ satisfy the necessary conditions in Definition \ref{def:regular}.
\end{ex}

As suggested by the name, any strongly branching Turing machine is branching. The graph interpretation is crucial to prove this property.

\begin{proposition}
A strongly branching Turing machine is branching.
\end{proposition}

\begin{proof}
Let $T$ be a strongly branching Turing machine with $\varepsilon=1$ (the other case is analogous), and let $Q'\subset Q\backslash\{q_{halt}\}$ and $\Sigma'\subset \Sigma$ be the subsets such that $\delta_{\varepsilon}|_{Q'\times \Sigma'}=1$ and $\delta_Q(Q'\times \Sigma')\subseteq Q'$ with $|\Sigma'|\geq 2$. Consider the associated graph defined above (with $\varepsilon=1$) and the subgraph $G$ given by the vertices $Q'$. It is clear that each vertex of $G$ is the origin of at least two edges, since any pair $(q,s)\in Q'\times\Sigma'$ has $\delta_{\varepsilon}(q,s)=1$ and $|\Sigma'|\geq 2$. Starting with a vertex $q_1\in Q'$ of $G$, we can iteratively move along an edge (without ever repeating that edge), following a sequence of vertices $q_i\in Q'$ and stop whenever we reach a $k$ such that $q_k=q_j$ for some $j<k$. This will necessarily happen, since after we have moved $|Q'|-1$ times, we will repeat a vertex.  This way we find a cycle $C_1$, and we denote its set of vertices by $V_1$.

Consider the graph $G_1$ obtained by removing from $G$ the edges of $C_1$, and take a vertex $q\in Q'\setminus V_1$. Again, iteratively move along the edges of the graph starting at $q$. Notice that each vertex in $V_1$ is the origin of some edge, hence after at most $|Q|-1$ steps we will find another cycle $C_2$ with vertices $V_2$. If $V_1\cap V_2\neq \emptyset$, we are done. Otherwise $V_1$ and $V_2$ are disjoint, and we consider the graph $G_2$ obtained by removing from $G_1$ the edges of the cycle $C_2$. Repeating this process, if we do not find two cycles sharing a vertex, we end up with disjoint cycles $C_1,...,C_N$ containing every vertex of $G$. If we remove all the edges of the cycles $C_1,...,C_N$ from $G$, we obtain a graph $G'$ such that every vertex is the origin of at least one edge. We can apply the argument once more to $G'$, finding another cycle $C_0$, which necessarily intersects one of the $C_1,...,C_N$, which completes the proof of the proposition.
\end{proof}

Finally, we are ready to prove Theorem~\ref{T.main}.

\begin{theorem}\label{thm:regular}
A branching Turing machine has positive topological entropy.
\end{theorem}
\begin{proof}
Let $T$ be a branching Turing machine, and let us assume that $\varepsilon=1$, the other case being analogous.
Define the integers
$$a_1:= 1+\sum_{i=1}^{m_1} \tau(q_i,s_i) , \quad a_2:= 1+ \sum_{i=1}^{m_2} \tau(q_i',s_i'), $$
and assume without any loss of generality that $a:=a_1\geq a_2$. Obviously, $a_1> m_1$ and $a_2> m_2$, so $a> \max\{m_1,m_2\}$. Consider integers of the form $n=ra$ with $r\in \mathbb{N}_0$. We claim that
\begin{equation}\label{eq:exp}
|S(n, R_T)| \geq 2^r.
\end{equation}
It is then easy to check that the topological entropy of $T$ is positive. Indeed, using Oprocha's formula we have:
\begin{align*}
h(T)&=\lim_{n\rightarrow \infty} \frac{1}{n}\log|S(n,R_T)| \\
&=\lim_{r\rightarrow \infty} \frac{1}{n}\log|S(n,R_T)| \\
&\geq \lim_{r\rightarrow \infty} \frac{1}{ra}r\log 2\\
&=\frac{\log 2}{a}>0\,.
\end{align*}
To see that the estimate~\eqref{eq:exp} holds, we need to define some sequences of pairs in $Q\times \Sigma$. For each $(q_i,s_i)$, we define $(q_{i,1},s_{i,1})=\big(\delta_Q(q_i,s_i),\delta_\Sigma(q_i,s_i)\big)$ and then iteratively
$$(q_{i,j},s_{i,j})= \big(\delta_Q(q_{i,{j-1}},s_{i,{j-1}}\big), \delta_\Sigma(q_{i,{j-1}},s_{i,{j-1}})), \enspace \text{for } j\in \{2,...,\tau(q_i,s_i)-1\}.$$
We define analogously $(q_{i,j}',s_{i,j}')$. Consider the sequences
\begin{multline*}
u_1= \Big( \big(q_1,s_1\big),\big(q_{1,1},s_{1,1}\big), \big(q_{1,2},s_{1,2}\big),...,\big(q_{1,\tau(q_1,s_1)-1},s_{1,\tau(q_1,s_1){-1}}\big),
\big(q_2,s_2\big),\\ \big(q_{2,1},s_{2,1}\big), ..., \big(q_{m_1-1,\tau(q_{m_1-1},s_{m_1-1})-1},{s_{m_1-1,\tau(q_{m_1-1},s_{m_1-1})-1}}\big), \big(q_{m_1},s_{m_1}\big)\Big),
\end{multline*}
\begin{multline*}
u_2= \Big( \big(q_1',s_1'\big),\big(q_{1,1}',s_{1,1}'\big), \big(q_{1,2}',s_{1,2}'\big),...,\big(q_{1,\tau(q'_1,s'_1)-1}',s_{1,\tau(q'_1,s'_1)-1}' \big),\big(q_2',s_2'\big),\\ \big(q_{2,1}',s_{2,1}'\big), ..., \big(q_{m_1-1,\tau(q'_{m_1-1},s'_{m_1-1})-1}', s_{m_1-1,\tau(q'_{m_1-1},s'_{m_1-1})-1}'\big), \big(q_{m_1}',s_{m_1}'\big)\Big).
\end{multline*}
and any sequence of the form
$$u= v_1 \oplus v_2 \oplus ... \oplus  v_r, $$
where $v_i$ is equal to either $u_1$ or $u_2$. This sequence has size at most $ra$, and there are $2^r$ possible choices which are all different thanks to the property that the two sequences of pairs $(q_i,s_i)$ and $(q_i',s_i')$ in $Q\times \Sigma$ satisfy that each one is not a concatenation of copies of the other one. For each possible $u$, consider the initial tape
$$t_u=...00.t_1t_2...t_{ra}00... $$
constructed as follows. If $v_1=u_1$, then the first $m_1$ symbols of the tape are $s_1,...,s_{m_1}$. If $v_1=u_2$, then instead the first $m_2$ symbols of the tape are $s_1',...,s_{m_2}'$. The next group of symbols is determined by $v_2$, if $v_2=u_1$ then the next symbols are $s_1,...,s_{m_1}$, and if $v_2=u_2$ then the next symbols are $s_1',...,s_{m_2}'$. We do this up to $v_r$, and this determines at most $ra$ symbols. We can fill the rest of symbols up to $t_{ra}$ with zeroes, for instance. By construction, initializing the machine with the configuration $x=(q_1,t_u)$, it is easy to check that
$(R_T^i(x)_q, R_T^i(x)_0)$
follows sequentially the pairs in $u$, and after that it possibly runs through some other pairs in $Q\times \Sigma$. This shows that for each $u$ there is (at least) a distinct element in $S(n,R_T)$, which proves that the bound~\eqref{eq:exp} holds, as we wanted to show.
\end{proof}

Notice that, using the graph interpretation of a branching Turing machine, the proof of Theorem~\ref{thm:regular} yields that the number of paths stemming from the vertex that belongs to two different cycles grows exponentially with respect to the length of the path. Heuristically, this might be interpreted as a hint that the topological entropy is positive, as rigorously established in Theorem~\ref{thm:regular}.

We believe that the hypothesis that implies positive topological entropy, in our case, the definition of branching Turing machine, can probably be relaxed a bit at the cost of adding quite some more technical details and definitions. We have not found any example in the literature of a universal Turing machine that is not branching, although examples can certainly be constructed {artificially}, as explained in Section~\ref{appendix}.

\section{Topological entropy of Turing complete area-preserving diffeomorphisms and Euler flows}\label{S.app}

Our previous criterion implies {that, under continuous encodings as in \cite{CMPP2}, Turing complete dynamics obtained from simulating branching universal machines ``inherit" positive topological entropy from their Turing complete behavior. To apply in the case of the 3D Euler flows constructed in \cite{CMPP2},} we first recall the definition of a Turing complete dynamical system $X$ on a topological space $M$:

\begin{definition}[Turing complete dynamical system]
\label{def:TCDS}
A dynamical system $X$ on $M$ is \emph{Turing complete} if there is a universal Turing machine $T_u$ such that for any input $t_{in}$ of $T_u$, there is a computable point $p\in M$ and a computable open set $U\subset M$ such that $\operatorname{Orb}_X(p)\cap U\neq \emptyset$ if and only if the machine $T_u$ halts with input~$t_{in}$.
\end{definition}


\begin{remark}
{Although it is not explicitly stated in the definition, the computability of $p$ and $U$ in this context should not be understood in the sense of computable analysis. Instead, it can be defined by saying that there is a Turing machine that outputs the exact coordinates of $p$ in some coordinate system in finite time. 
Similarly, the open set $U$ should have a finite description. A way of making this precise, for example, when $M$ is a manifold, is as follows: for some fixed atlas of $M$, we require $p$ to be a rational point in some chart, and require $U$ to be a finite union of balls with rational centers and radii. Variations of definitions of computability involve allowing the use of $\operatorname{BSS}_{\operatorname{C}}$ machines, see \cite{CR}.} 

{
In addition, for the simulation to occur within the dynamics and not within the maps computing $p$ and $U$, one requires that the maps assigning to an input of the machine the corresponding point and open set should be ``reasonable". For instance, one can require that the coordinates of $p$ and the balls describing $U$ assigned to an input of the machine can be computed in a time that only depends on the size of the input of the machine. We refer to \cite{CR} for a formalization of this condition. In any case, this (and more, see the next paragraph) is satisfied by most constructions of Turing complete systems in the literature (including the ones we allude to in this section), as there the maps giving $p$ and $U$ are usually given in an explicit closed form. }
\end{remark}
{Let us also mention that, even though Definition~\ref{def:TCDS} is one of the most general definitions of Turing completeness, in practice, when constructing Turing complete systems, they often satisfy stronger simulation properties like some kind of simulation ``step by step''.} These simulation properties can be used to give stronger definitions of Turing completeness.
An example of such a property could be, for example, for a diffeomorphism $f:M\rightarrow M$ of a manifold, to require that there is a compact invariant subset $C\subset M$ such that $f|_C$ is conjugate or semiconjugate to the global transition function of the Turing machine $T_u$. In this direction, the following lemma follows from a combination of~\cite[Lemma 4.5]{CMPP2} and~\cite[Proposition 5.1]{CMPP2}:

\begin{lemma}
Let $T$ be a reversible Turing machine whose global transition function has been extended to halting configurations via the extension of the transition function $\delta(q_{halt},t)=(q_0,t,0)$ for all $t\in \Sigma$. Then there exists a bijective generalized shift $\Delta$ that is conjugate to $R_T$, and a smooth area-preserving diffeomorphism $\varphi:D\longrightarrow D$ of a disk $D$ (of radius larger than one) that is the identity near the boundary and whose restriction to the square Cantor set $C^2\subset D$ is conjugate to $\Delta$ by the homeomorphism $e$ in Equation~\eqref{eq:coding}.
\end{lemma}
It was shown in~\cite[Corollary 3.2]{CMPP2} that the diffeomorphism $\varphi$ can be realized as the first-return map on a disk-like transverse section of a stationary solution to the Euler equations for some metric on any compact three-manifold $M$, which yields a Turing complete stationary fluid flow~\cite[Theorem 6.1]{CMPP2}. The main application of Theorem~\ref{T.main} is that the Turing complete Euler flows constructed in~\cite{CMPP2} always have positive topological entropy whenever the simulated universal Turing machine is branching.
\begin{corollary}
If $T$ is a branching reversible Turing machine then its associated diffeomorphism of the disk $\varphi$ and the steady Euler flows on a compact three-manifold $M$ constructed in~\cite{CMPP2} have positive topological entropy. If $T$ is universal, then $\varphi$ exhibits a compact chaotic invariant set $K$ homeomorphic to the square Cantor set so that $\varphi|_K$ is Turing complete.
\end{corollary}
\begin{proof}
Given Theorem~\ref{thm:regular}, the Turing machine $T$ has positive topological entropy. Its associated generalized shift $\Delta$ has positive topological entropy too by Proposition~\ref{prop:entropyGS}. The identification of the space of sequences of the generalized shift with a square Cantor set $C\subset [0,1]^2$ via the homeomorphism~\eqref{eq:coding} implies that $\Delta$ induces a map $\tilde \Delta: C\rightarrow C$. By~\cite[Proposition 5.1]{CMPP2} there exists an area-preserving diffeomorphism $\varphi:D\rightarrow D$ of a disk $D$ strictly containing the unit square whose restriction to $C$, which is an invariant set, coincides with $\tilde \Delta$. Hence, the topological entropy of $\varphi$ is necessarily positive. The compact chaotic invariant set is $K\equiv C$, and $\varphi|_K$ is Turing complete whenever $T$ is universal. The stationary Euler flow in~\cite[Theorem 6.1]{CMPP2} admits a transverse disk where the first-return map is conjugate to $\varphi$, and therefore it has positive topological entropy too.
\end{proof}

\section{A universal Turing machine with zero topological entropy}\label{appendix}

\newcommand{\N}{\mathbb{N}}

As we have shown by examples, most ``natural'' universal Turing machines in the literature are branching. We now show that it is nevertheless possible to construct a universal Turing machine with zero topological entropy (thus, in particular, not branching), although it does not seem likely that such an object would arise from a natural example.
\subsection{The example: proof of Theorem~\ref{T.main2}}
The construction is based directly on the universal counter machines of Minsky \cite{M}. After this construction, we explain how to obtain another proof (of a stronger result) from a more involved construction of Hooper \cite{H} (or Kari and Ollinger \cite{KM}). 

\begin{theorem}
There is a universal Turing machine with zero topological entropy.
\end{theorem}

\begin{proof}
The idea of the construction is to exponentially slow down the computation of an arbitrary Turing machine $T$, so that universality is retained, but the resulting machine $U$ has zero entropy (even if $T$ does not). We do this by stacking two standard simulations: Minsky's simulation of a Turing machine by a counter machine, and the simulation of a counter machine by a Turing machine.

We recall the definition of a $2$-counter machine. This is $C = (Q, q_0, q_{\text{halt}}, \delta)$ where $\delta : (Q \setminus \{q_{\text{halt}}\}) \times \{0, +\}^2 \to Q \times \{-1, 0, +1\}^2$ is the transition relation. The machine defines a partial transition function on $Q \times \N \times \N$ (defined if and only if the state is not $q_{\text{halt}}$). The interpretation of $(q, m, n)$ is that the machine is in state $q$, and the current counter values are $m, n$.

The interpretation of $\delta(q, a_1, a_2) = (q', b_1, b_2)$ is that if the current state is $q$ and $a_i = 0$ iff the $i$th counter has zero value, then we step into state $q'$ and add $(b_1, b_2)$ to the counter values (we require that $a_i = 0 \implies b_i \neq -1$). Write $t \Rightarrow_C t'$ if there is a one-step computation of $C$ from $t \in Q \times \N \times \N$ to $t' \in Q \times \N \times \N$, and as in the case of Turing machines define a partial function $\Phi_C(t) = t'$ if $t \Rightarrow_C^* t'$ and $t'$ is halting (i.e.,\ $t' = (q_{\text{halt}}, m, n)$ for some $m, n$).

It is a result of Minsky \cite{M} that a 2-counter machine can simulate an arbitrary Turing machine, in the sense that for any Turing machine $T$ with states $Q$, we can find a 2-counter machine with states $Q' \supset Q$ and the same halting state $q_{\mathrm{halt}}$, and total computable functions $e : X_T^c \to Q \times \N$ and ${f} : Q \times \N \to X_T^c$ such that if started from state $(e(t), 0)$ with $t \in X_T^c$, the next time we are in a state $(q', m', 0) \in Q \times \N \times \N$ (i.e., the next time the second counter contains $0$, and the state is in the subset $Q$ of $Q'$, we have ${f}(q', m') = t'$ such that $t \Rightarrow_T t'$). 

Next, for any counter machine $C$ with states $Q$ and transitions $\delta$, it is again possible to construct a Turing machine $T$ with alphabet $\{@, 0, 1\}$ and states $Q' = Q \sqcup Q''$ which simulates the counter machine in the following sense: If started on the configuration
\[ {^\omega} 0 1 0^m @ 0^n 1 0^\omega \]
with the head on the $@$-symbol in state $q \in Q$, then when in the sequence of $\Rightarrow_T$-steps we next enter a configuration of the form
\[ {^\omega} 0 1 0^{m'} @ 0^{n'} 1 0^\omega \]
with the head on the $@$-symbol in some state $q' \in Q$, we have $(q, m, n) \Rightarrow_C (q', m', n')$. We call this the \emph{simulation property}.

The way this is done is simply that the Turing machine performs a back-and-forth sweep both ways to check which of the counters have zero value, and then performs new sweeps in order to update them, according to the transition function of $C$. 

If we start with any universal Turing machine $T$, then simulate it with a 2-counter machine $C$ in the sense of Minsky, and then again simulate it with another Turing machine $U$ with the simulation property, then $U$ is also universal: given an input $x \in \mathcal{X}$ and the number $n$ describing the Turing machine to simulate, we first use the universality of $T$ to find an input to give to $T$. This is then converted to an input for $C$, and then to one for $U$. As for the decoding function $d$, if $U$ halts, from the simulation properties we can read off how the simulated machine $C$ halted, and from that how the simulated machine $T$ halted, and lastly from this we recover how $T_n$ halts on input $x$.

We now describe a naive concrete implementation of such a machine $U$ (simulating any counter machine, such as one simulating a universal Turing machine).

One can use a state set of the form $Q \sqcup (Q \times \{0, +\}^2 \times R) \sqcup \{\bot\}$, where $Q$ simulates the states of the counter machine, $\bot$ is a fail state, and the elements of $Q \times \{0, +\}^2 \times D$ are interpreted as follows: $Q$ remembers the counter machine state, $\{0, +\}^2$ is a finite amount of memory for storing values of counters, and $R$ is used to remember what we are doing. 

Specifically, we can pick
\begin{align*}R = \{&\text{left check}, \text{left check return}, \text{right check}, \text{right check return}, \text{left update}, \\ &\text{left drop}, \text{left update return}, \text{right update}, \text{right drop}, \text{right update return}\}. \end{align*}

Initially, when initialized in a state simulating a state of the counter machine, our machine does not know what the counter values are, and should perform two sweeps in order to calculate them. This can be done as follows: 
\begin{align*}
\delta'(q, @) &= ((q, 0, 0, \text{left check}), @, -1) \\
\delta'((q, a, 0, \text{left check}), 0) &= ((q, +, 0, \text{left check}), 0, -1) \\
\delta'((q, a, 0, \text{left check}), 1) &= ((q, a, 0, \text{left check return}), 1, 1) \\
\delta'((q, a, 0, \text{left check return}), 0) &= ((q, a, 0, \text{left check return}), 0, 1) \\
\delta'((q, a, 0, \text{left check return}), @) &= ((q, a, 0, \text{right check}), @, 1) \\
\delta'((q, a, b, \text{right check}), 0) &= ((q, a, +, \text{right check}), 0, 1) \\
\delta'((q, a, b, \text{right check}), 1) &= ((q, a, b, \text{right check return}), 1, -1) \\
\delta'((q, a, b, \text{right check return}), 0) &= ((q, a, b, \text{right check return}), 0, -1) \\
\delta'((q, a, b, \text{right check return}), @) &= ((q, a, b, \text{left update}), @, -1)
\end{align*}
At this point, the state is expected to be $(q,a,b,\text{left update})$ where $q$ is the simulated counter machine state, and $a, b \in \{0, +\}$ are the information about whether the counters are zero or positive. Next, we should update the counters. Note that decrementing the left counter means simply changing the $1$ on the left to $0$, and rewriting it on the right (by using the drop state), and similarly for other counter updates; and symmetrically for the right counter.

Specifically, assume the counter machine transitions as $\delta(q, a, b) = (q', c, d)$. Then the added transitions can be taken to be
\begin{align*}
\delta'((q, a, b, \text{left update}), 0) &= ((q, a, b, \text{left update}), 0, -1) \\
\delta'((q, a, b, \text{left update}), 1) &= ((q, a, b, \text{left drop}), 0, -c) \\
\delta'((q, a, b, \text{left drop}), 0) &= ((q, a, b, \text{left update return}), 1, 1) \\
\delta'((q, a, b, \text{left update return}), 0) &= ((q, a, b, \text{left update return}), 0, 1) \\
\delta'((q, a, b, \text{left update return}), @) &= ((q, a, b, \text{right update}), @, 1) \\
\delta'((q, a, b, \text{right update}), 0) &= ((q, a, b, \text{right update}), 0, 1) \\
\delta'((q, a, b, \text{right update}), 1) &= ((q, a, b, \text{right drop}), 0, d) \\
\delta'((q, a, b, \text{right drop}), 0) &= ((q, a, b, \text{right update return}), 1, -1) \\
\delta'((q, a, b, \text{right update return}), 0) &= ((q, a, b, \text{right update return}), 0, -1) \\
\delta'((q, a, b, \text{right update return}), @) &= (q', @, 0)
\end{align*}
It is clear that no matter what the other transitions are, this realizes the counter machine simulation correctly: one can exactly calculate the sequence of moves performed on the configuration ${^\omega} 0 1 0^m @ 0^n 1 0^\omega$, and no unexpected situations can arise (note that by assumption, our counter machines never try to decrement a counter with value zero).

Now we let all other transitions enter the state $\bot$, and in this state, loop forever without moving. We claim that then the machine has zero topological entropy. For this, we analyze a computation of the machine on an arbitrary configuration, for $N$ steps. 

Observe that when we defined $R$ above, we listed it in a particular order. Our machine has the property that when it is in a state outside $Q \cup \{\bot\}$, the $R$-component of the state will evolve in this order, until the machine has either entered $\bot$ (and is in an infinite loop without moving), or is back to a state of $Q$, necessarily on top of the symbol $@$. From a quick look at the transitions, we see that it moves to the next state of $R$ whenever it sees any nonzero symbol on the tape. Note that this implies that this initial segment of the computation can be described by at most four numbers and some constant information, thus has a description with $\log O(N^4) = O(\log N)$ bits.

Once we are in state $Q$, the computation in fact simulates the counter machine exactly as above, or enters the state $\bot$: The machine will look for the next $1$ to the left of $@$, then for the next $1$ to the right, and then update their positions. If it never finds such $1$, or runs into another $@$-symbol, it enters state $\bot$; or if it tries to move $1$ further away from $@$, then that position must contain $0$ or it enters state $\bot$.

We use some tools from symbolic dynamics to compute the entropy (a direct calculation would also be straightforward, but we prefer this proof as it better elucidates why the long traversals automatically imply zero entropy). First recall Oprocha's formula~\eqref{eq.opr} stating the entropy in terms of the words $S(N, U)$ \cite{Op}. We recall the symbolic dynamical interpretation of this: Define the set $Z \subset (Q\times \Sigma)^\omega$ of all infinite words whose finite $n$-prefix is in $S(n, U)$ for all $n$. Then $Z$ is a subshift, meaning it is topologically closed and closed under the left shift $\sigma(z)_i = z_{i+1}$. Oprocha's formula states that the entropy of $U$ is equal to the entropy of $Z$, and thus it suffices to show that $Z$ has zero entropy.

It is well-known that positive entropy for a subshift $Z$ implies that some infinite configuration $z \in Z$ has linear Kolmogorov complexity in all prefixes \cite{S,J}, meaning the shortest program that generates the prefix of length $n$ of $z$ from empty input has length $\Theta(n)$. Thus if we had positive entropy for the Turing machine, then some words in $S(N, U)$ would require at least $C N$ bits to describe for all large enough $N$. Suppose this is the case, and we show a contradiction by compressing them strictly more efficiently.

By the explanation above, any computation can be compressed by remembering $O(\log N)$ bits; then the part where we simulate the counter machine; then another $O(\log N)$ bit compressible suffix describing a computation that does not cycle through all of $R$; and finally possibly we remember a number indicating how much time we spend in state $\bot$, again requiring at most $O(\log N)$ bits.

We now analyze the part where we simulate the counter machine, as it is the only possible source of a linear amount of Kolmogorov complexity. Let $(q_1, m_1, n_1), (q_2, m_2, n_2), \ldots$ be the sequence of simulated states and counter values encountered during this simulation part. Note that we must have $N \geq \sum_i (m_i + n_i)$, as the machine certainly spends more than $m_i + n_i$ steps to read counter values $m_i$ and $n_i$ encoded in the distances of $1$s from the $@$-symbol, and to update them (recall that we always read these values whether or not the machine actually ``needs'' to know their values).

We can compress the information about this sequence into $B \sum_i \log(m_i + n_i)$ bits for some constant $B$, by simply writing down the numbers in binary (more naturally we get $\log(m_i) + \log(n_i)$, but $\log(m_i) + \log(n_i) \leq 2\log(m_i + n_i)$ and the $2$ disappears into $B$). Let $I$ be the set of $i$ such that $\log (m_i + n_i) < \frac{C}{2B}(m_i + n_i)$ Note that this is true whenever $m_i + n_i > D$ for some constant $D$. Let $J$ be the complement of $I$ (among indices of the $(q_i, m_i, n_i)$).

If we do not have a repetition among the $(q_i, m_i, n_i)$, then have the (rough) upper bound
\begin{align*}
B \sum_i \log(m_i + n_i) &\leq C\sum_{i \in I} (m_i + n_i)/2 + B \sum_{i \in J} \log D \\
&\leq C N/2 + B |Q| D^2 \log D
\end{align*}
where $B, |Q|, D$ do not depend on $N$, so this is far smaller than $CN$ for large $N$, even together with the initial and final parts of the computation that took $O(\log N)$ bits to compress.

On the other hand, computations with repeated $(q_i, m_i, n_i)$ are periodic, and even easier to compress.

This contradiction proves that there cannot be a linear lower bound on the compressibility, which finally concludes the proof of zero entropy.
\end{proof}

\subsection{The speed of the Turing machine}

A Turing machine admits a notion of \emph{speed}, namely one calculates the maximal offset by which the head can move in $n$ steps. We observe that this quantity is subadditive and takes a normalized limit using Fekete's lemma. It is easy to show that zero speed implies zero entropy.

In 1969, Hooper proved~\cite{H} (see \cite{KM} for a reversible version with an arguably easier proof) that given a Turing machine, it is undecidable whether it admits configurations where the machine never halts. If a Turing machine halts on every configuration, then a simple compactness argument shows that there is a bound on the number of steps it takes to halt. Thus, the undecidability must come from computations that do not halt. 

Thus, Hooper at least had to show that one can perform universal computation 
with a Turing machine such that there are no situations where it is easy to prove that the machine never halts (on infinite configurations). One such situation is an infinite ``search'' for a symbol. In all direct simulations (and definitely in the counter machine simulation we performed above), there are such infinite searches, and due to compactness of the configuration space, it is tempting to think that they are necessary. They are not, and the genius trick of Hooper was to show that one can trick compactness by starting computations recursively, so that even though there are infinite searches, there are other searches between them. This is analogous to Berger's proof in \cite{B} of the undecidability of the domino problem.

It was later clarified by Jeandel that Hooper was in a sense literally fighting positive speed: \cite{J} shows that if a Turing machine has positive speed (resp.\ positive entropy), then this can be proved in ZF.

Thus, our conclusion is that at least infinitely many of Hooper's machines must have zero speed, thus zero entropy. Since they involve an undecidability problem, one should expect them to involve universal computation, and indeed Hooper's machines have literally the simulation property we described above (except the encoding is somewhat different, and there are many intermediate configurations where multiple @-symbols appear, due to the recursive computations started at all times).

Unfortunately, Hooper was not explicitly concerned with universal Turing machines, nor explicitly discusses speed or entropy, and thus we did not find it easy to use his results as a black box to prove even the existence of a zero entropy universal Turing machine. Nevertheless, there is no doubt that his construction implies that zero speed universal Turing machines exist, and as we have tried to argue here this is morally an automatic consequence of his result.

We state the stronger reversible statement: A reversible variant of Hooper's construction is given in \cite{KM}. This is also a direct simulation of a reversible counter machine, and such a machine can simulate an arbitrary (not necessarily reversible) Turing machine up to a computable encoding. From the construction, one thus obtains the following result (stated as Theorem \ref{T.main2} in the Introduction).

\begin{theorem}
There exists a universal Turing machine that is reversible and has zero speed. In particular, it has zero topological entropy.
\end{theorem}

\section*{Acknowledgements}
The authors are very grateful to Leonid Polterovich for suggesting that we study the topological entropy of Turing-complete dynamical systems, and to the reviewers of this article for their suggestions to improve the manuscript.

\end{document}